\documentclass[12pt, reqno]{amsart}
 
\usepackage{amssymb, setspace, bbm}
\usepackage[colorlinks, backref=page, citecolor=blue]{hyperref}
\usepackage[backrefs, alphabetic]{amsrefs}
\usepackage[margin=1.1in]{geometry}
\numberwithin{equation}{section}
\theoremstyle{plain}
\newtheorem{thm}{Theorem}[section]

\newtheorem{prop}{Proposition}[section]
 
\theoremstyle{definition}
\newtheorem{defin}{Definition}[section]
\newtheorem{qst}{Question}
\newtheorem{remark}{Remark}[section]

\newcommand{\R}{\mathbb R}

\begin{document}
 
\title[nearest points and DC functions]{nearest points and delta convex functions in Banach spaces}

\author{Jonathan M. Borwein \and Ohad Giladi}
\address{Centre for Computer-assisted Research Mathematics and its Applications (CARMA), School of Mathematical and Physical Sciences, University of Newcastle, Callaghan, NSW 2308, Australia}
\email{jonathan.borwein@newcastle.edu.au, ohad.giladi@newcastle.edu.au}

\begin{abstract}
Given a closed set $C$ in a Banach space $(X, \|\cdot\|)$, a point $x\in X$ is said to have a nearest point in $C$ if there exists $z\in C$ such that $d_C(x) =\|x-z\|$, where $d_C$ is the distance of $x$ from $C$. We shortly survey the problem of studying how large is the set of points in $X$ which have nearest points in $C$. We then discuss the topic of delta-convex functions and how it is related to finding nearest points. 
\end{abstract}

 
 
\maketitle
 
\section{Nearest points in Banach spaces}
 
\subsection{Background}\label{sec intro}
 
Let $(X, \|\cdot\|)$ be a real Banach space, and let $C\subseteq X$ be a non-empty closed set. Given $x\in X$, its distance from $C$ is given by
\[d_C(x) = \inf_{y\in C} \|x-y\|.\]
If there exists $z\in C$ with $d_C(x) = \|x-z\|$, we say that $x$ has a \emph{nearest point} in $C$. Let also
\[N(C) = \big\{x\in X: x \text{ has a nearest point in $C$ }\big\}.\]
One can then ask questions about the structure of the set $N(C)$. This question has been studied in \cite{Ste63, Lau78, Kon80, Zaj83, BF89, DMP91, Dud04, RZ11, RZ12} to name just a few. More specifically, the following questions are at the heart of this note:

\medskip
\begin{center}
\emph{Given a nonempty closed set $C\subseteq X$, how large is the set $N(C)$? When is it non-empty?}
\end{center}
\medskip

One way to do so is to consider sets which are large in the set theoretic sense, such as dense $G_{\delta}$ sets. We begin with a few definitions.

\begin{defin}
If $N(C) =  X$, i.e., every point in $X$ has a nearest point in $C$, then $C$ is said to be proximinal. If $N(C)$ contains a dense $G_{\delta}$ set, then $C$ is said to be almost proximinal.
\end{defin}

In passing we recall that If every point in $X$ is uniquely proximinal then $C$ is said to be a \emph{Chebyshev set}  It has been conjectured for over half a century,  that in Hilbert space 
Chebyshev sets are necessarily convex, but this is only proven for weakly closed sets \cite{BV10}. See also \cite{FM15} for a recent survey on the topic.

For example, closed convex sets in reflexive spaces are proximinal, as well as closed sets in finite dimensional spaces. See \cite{BF89}. One can also consider stronger notions of ``large" sets. See Section \ref{sec porous}. First, we also need the following definition.

\begin{defin}
A Banach space is said to be a (sequentially) Kadec space if for each sequence $\{x_n\}$ that converges weakly to $x$ with $\lim \|x_n\| = \|x\|$, $\{x_n\}$ converges to $x$ in norm, i.e.,
\[\lim_{n\to \infty} \|x-x_n\| =0.\]
\end{defin}

With the above definitions in hand, the following result holds.

\begin{thm}[Lau \cite{Lau78}, Borwein-Fitzpatrick \cite{BF89}]\label{thm Lau BF}
If $X$ is a reflexive Kadec space and $C\subseteq X$ is closed, then $C$ is almost proximinal.
\end{thm}

The assumptions on $X$ are in fact necessary.

\begin{thm}[Konjagin \cite{Kon80}]\label{thm Kon}
If $X$ is not both Kadec and reflexive, then there exist $C\subseteq X$ closed and $U\subseteq X\setminus C$ open such that no $x\in U$ has a nearest point in $C$.
\end{thm}

It is known that under stronger assumption on $X$ one can obtain stronger results on the set $N(C)$. See Section \ref{sec porous}.

\medskip

\subsection{Fr\'echet sub-differentiability and nearest points}\label{sec diff}
 
We begin with a definition.

\begin{defin}
Assume that $f: X\to \R$ is a real valued function with $f(x)$ finite. Then $f$ is said to be Fr\'echet sub-differentiable at $x\in X$ if there exists $x^*\in X^*$ such that
\begin{align}\label{cond subdiff}
\liminf_{y\to 0}\frac{f(x+y)-f(x)-x^*(y)}{\|y\|} \ge 0.
\end{align}
The set of points in $X^*$ that satisfy \eqref{cond subdiff} is denoted by $\partial f(x)$.
\end{defin} 

Sub-derivatives have been found to have many applications in approximation theory. See for example \cite{BF89, BZ05, BL06, BV10, Pen13}. 
 
One of the connections between sub-differentiability and the nearest point problem was studied in \cite{BF89}. Given $C\subseteq X$ closed, the following modification of a construction of \cite{Lau78} was introduced.
\begin{align*}
L_n(C) = \Big\{x\in X\setminus C : \exists x^*\in \mathbb S_{X^*} \text{ s.t. }\sup_{\delta>0}~\inf_{z\in C\cap B(x,d_C(x)+\delta)}~x^*(x-z) > \big(1-2^{-n}\big)d_C(x)\Big\},
\end{align*}
where $\mathbb S_{X^*}$ denotes the unit sphere of $X^*$. Also, let
\begin{align*}
L(C) = \bigcap_{n=1}^{\infty}L_n(C).
\end{align*}
The following is known.
\begin{prop}[Borwein-Fitzpatrick \cite{BF89}]\label{prop open}
For every $n\in \mathbb N$, $L_n(C)$ is open. In particular, $L(C)$ is $G_{\delta}$.
\end{prop}
Finally, let
\begin{align*}
\Omega(C)  = & \Big\{ x\in X\setminus C : \exists x^*\in \mathbb S_{X^*}, \text{ s.t. } \forall \epsilon >0, \exists \delta>0,
\\ & ~~\quad\inf_{z\in C\cap B(x,d_C(x)+\delta)}x^*(x-z) > \big(1-\epsilon\big)d_C(x)\Big\}.
\end{align*}

While $L(C)$ is $G_{\delta}$ by Proposition \ref{prop open}, under the assumption that $X$ is reflexive, the following is known.

\begin{prop}[Borwein-Fitzpatrick \cite{BF89}]
If $X$ is reflexive then $\Omega(C) = L(C)$. In particular, $\Omega(C)$ is $G_{\delta}$.
\end{prop}

The connection to sub-differentiability is given in the following proposition.

\begin{prop}[Borwein-Fitzpatrick \cite{BF89}]
If $x\in X\setminus C$ and $\partial d_C(x) \neq \emptyset$, then $x\in \Omega(C)$.
\end{prop}

Also, the following result is known.

\begin{thm}[Borwein-Preiss \cite{BP87}]\label{thm var}
If $f$ is lower semicontiuous on a reflexive Banach space, then $f$ is Fr\'echet sub-differentiable on a dense set.
\end{thm}

In fact, Theorem \ref{thm var} holds under a weaker assumption. See \cite{BP87, BF89}.
Since the distance function is lower semicontinuous, it follows that it is sub-differentiable on a dense subset, and therefore, by the above propositions, $\Omega(C)$ is a dense $G_{\delta}$ set. Thus, in order to prove Theorem \ref{thm Lau BF}, it is only left to show that every $x\in \Omega(C)$ has a nearest point in $C$. Indeed, if $\{z_n\}\subseteq C$ is a minimizing sequence, then by extracting a subsequence, assume that $\{z_n\}$ has a weak limit $z\in C$. By the definition of $\Omega(C)$, there exists $x^* \in \mathbb S_{X^*}$ such that
\[\|x-z\| \ge x^*(x-z) = \lim_{n\to \infty} x^*(x-z_n) \ge d_C(x) = \lim_{n\to \infty}\|x-z_n\|.\]
On the other hand, by weak lower semicontinuity of the norm,
\[\lim_{n\to \infty} \|x-z_n\| \ge \|x-z\|,\]
and so $\|x-z\| = \lim\|x-z_n\|$. Since it is known that $\{z_n\}$ converges weakly to $z$, the Kadec property implies that in fact $\{z_n\}$ converges in norm to $z$. Thus $z$ is a nearest point. This completes the proof of Theorem \ref{thm Lau BF}.

This scheme of proof from \cite{BF89} shows that differentiation arguments can be used to prove that $N(C)$ is large.

\medskip

\subsection{Nearest points in non-Kadec spaces}\label{sec non-Kadec}
It was previously mentioned that closed convex sets in reflexive spaces are proximinal. It also known that non-empty ``Swiss cheese" sets (sets whose complement is a mutually disjoint union of open convex sets) in reflexive spaces are almost proximinal \cite{BF89}. These two examples show that for some classes of closed sets, the Kadec property can be removed. Moreover, one can consider another, weaker, way to ``measure" whether a set $C\subseteq X$ has ``many" nearest points: ask whether the set of nearest points in $C$ to points in $X\setminus C$ is dense in the boundary of $C$. Note that if $C$ is almost proximinal, then nearest points are dense in the boundary. The converse, however, is not true. In \cite{BF89} an example of a non-Kadec reflexive space was constructed where for every closed set, the set of nearest points is dense in its boundary. The following general question is still open.

\begin{qst}
Let $(X, \|\cdot\|)$ be a reflexive Banach space and $C\subseteq X$ closed. Is the set of nearest points in $C$ to points in $X\setminus C$ dense in its boundary?
\end{qst}

 Relatedly,  if the set $C$ is norm closed and bounded in a space with the \emph{Radon-Nikodym property} as is the caae of reflexive space, then $N(C)$ is nonempty and is large enough so that $\overline{\rm conv} C = \overline{\rm conv} N(C)$ \cite{BF89}.

\medskip
 
\subsection{Porosity and nearest points}\label{sec porous}
 
As was mentioned in subsection \ref{sec diff}, one can consider stronger notions of ``large" sets. One is the following notion.
 
\begin{defin}
A set $S\subseteq X$ is said to be porous if there exists $c\in (0,1)$ such that for every $x\in X$ and every $\epsilon>0$, there is a $y \in B(0,\epsilon)\setminus \{0\}$ such that
\[ B(x+y, c\|y\|) \cap S = \emptyset.\]
A set is said to be $\sigma$-porous if it a countable union of porous sets. Here and in what follows, $B(x,r)$ denotes the closed ball around $x$ with radius $r$.
\end{defin}

See \cite{Zaj05, LPT12} for a more detailed discussion on porous sets. It is known that every $\sigma$-porous set is of the first category, i.e., union of nowhere dense set. Moreover, it is known that the class of $\sigma$-porous sets is a proper sub-class of the class of first category sets. When $X=\R^n$, one can show that every $\sigma$-porous set has Lebesgue measure zero. This is not the case for every first category set: $\R$ can be written as a disjoint union of a set of the first category and a set of Lebesgue measure zero. Hence, the notion of porosity automatically gives a stronger notion of large sets: every set whose complement is $\sigma$-porous is also a dense $G_{\delta}$ set.

\medskip

A Banach space $(X, \|\cdot\|)$ is said to be \emph{uniformly convex} if the function
\begin{align}\label{def uni conv}
\delta(\epsilon) = \inf\left\{ 1-\left\|\frac{x+y}2\right\| ~ : ~ x, y\in \mathbb S_X, \|x-y\| \ge \epsilon \right\},
\end{align}
is strictly positive whenever $\epsilon>0$. Here $\mathbb S_X$ denotes the unit sphere of $X$. In \cite{DMP91} the following was shown.
 
\begin{thm}[De Blasi-Myjak-Papini \cite{DMP91}]\label{thm DMP}
If $X$ is uniformly convex, then $N(C)$ has a $\sigma$-porous compliment.
\end{thm}

In fact, \cite{DMP91} proved a stronger result, namely that for every $x$ outside a $\sigma$-porous set, the minimization problem is \emph{well posed}, i.e., there is unique minimizer to which every minimizing sequence converges. See also \cite{FP91, RZ11, RZ12} for closely related results in this direction. 

The proof of Theorem \ref{thm DMP} builds on ideas developed in \cite{Ste63}. However, it would be interesting to know whether one could use differentiation arguments as in Section \ref{sec diff}. This raises the following question:

\begin{qst}
Can differentiation arguments be used to give an alternative proof of Theorem \ref{thm DMP}?
\end{qst}

More specifically, if one can show that $\partial d_C \neq \emptyset$ outside a $\sigma$-porous set, then by the arguments presented in Section \ref{sec diff}, it would follow that $N(C)$ has a $\sigma$-porous complement. Next, we mention two important results regarding differentiation in Banach spaces. 
 
\begin{thm}[Preiss-Zaj\'i\v{c}ek \cite{PZ84}]\label{thm PZ}
If $X$ has a separable dual and $f:X\to \R$ is continuous and convex, then $X$ is Fr\'echet differentiable outside a $\sigma$-porous set.
\end{thm}
 
See also \cite[Sec. 3.3]{LPT12}. Theorem \ref{thm PZ} implies that if, for example, $d_C$ is a linear combination of convex functions (see more on this in Section \ref{sec DC}), then $N(C)$ has a $\sigma$-porous complement. Also, we have the following.

\begin{thm}[C\'uth-Rmoutil \cite{CR13}]\label{thm CR}
If $X$ has a separable dual and $f:X\to \R$ is Lipschitz, then the set of points where $f$ is Fr\'echet sub-differentiable but not differentiable is $\sigma$-porous.
\end{thm}

Since $d_C$ is 1-Lipschitz, the questions of seeking points of sub-differentiability or points of differentiability are similar. Theorem \ref{thm PZ} and Theorem \ref{thm CR} remain true if we consider $f:A\to \R$ where $A\subseteq X$ is open and convex.

\bigskip

\section{DC functions and DC sets}\label{sec DC}
 
\subsection{Background}

\begin{defin}
A function $f:X\to \R$ is said to be delta-convex, or DC, if it can be written as a difference of two convex functions on $X$. 
\end{defin}

This notion was introduced in \cite{Har59} and was later studied by many authors. See for example \cite{KM90, Cep98, Dud01, VZ01, DVZ03, BZ05, Pav05, BB11}. In particular, \cite{BB11} gives a good introduction to this topic. We will discuss here only the parts that are closely related to the nearest point problem.
 
The following is an important proposition. See for example \cite{VZ89, HPT00} for a proof.
 
\begin{prop}\label{prop select}
If $f_1,\dots,f_k$ are DC functions and $f:X\to \R$ is continuous and $f(x)\in \big\{f_1(x),\dots,f_n(x)\big\}$. Then $f$ is also DC.
\end{prop}

The result is true if we replace the domain $X$ by any convex subset.

\medskip

\subsection{DC functions and nearest points}

Showing that a given function is in fact DC is a powerful tool, as it allows us to use many known results about convex and DC functions. For example, if a function is DC on a Banach space with a separable dual, then by Theorem \ref{thm PZ}, it is differentiable outside a $\sigma$-porous set. In the context of the nearest point problem, if we know that the distance function is DC, then using the scheme presented in Section \ref{sec diff}, it would follow that $N(C)$ has a $\sigma$-porous complement. The same holds if we have a difference of a convex function and, say, a smooth function.

\medskip
 
The simplest and best known example is when $(X,\|\cdot\|)$ is a Hilbert space, where we have the following.
 
\begin{align*}
d_C^2(x) & = \inf_{y\in C}\|x-y\|^2
\\ & = \inf_{y\in C}\Big[\|x\|^2-2\langle x,y\rangle +\|y\|^2\Big]
\\ & = \|x\|^2 - 2\sup_{y\in C}\Big[\langle x,y\rangle - \|y\|^2/2\Big],
\end{align*}
and the function $x\mapsto  \sup_{y\in C}\Big[\langle x,y\rangle - \|y\|^2/2\Big]$ is convex as a supremum of affine functions. Hence $d_C^2$ is DC on $X$.
Moreover, in a Hilbert space we have the following result (see \cite[Sec. 5.3]{BZ05}).

\begin{thm}\label{thm local DC}
If $(X, \|\cdot\|)$ is a Hilbert space, $d_C$ is locally DC on $X\setminus C$.
\end{thm}

\begin{proof}
Fix $y\in C$ and $x_0\in X\setminus C$. It can be shown that if we let $f_y(x) = \|x-y\|$, then $f_y$ satisfies
\begin{align*}
\big\|f_y'(x_1)-f_y'(x_2)\big\|_{X^*} \le L_{x_0}\|x_1-x_2\|, ~~ x_1,x_2 \in B_{x_0},
\end{align*}
where $L_{x_0} = \frac 4 {d_S(x_0)}$ and $B_{x_0} = B\Big(x_0, \frac 1 2 d_C(x_0)\Big)$. In particular,
\begin{align}\label{lip prop}
\big(f_y'(x+tv_1)-f_y'(x+t_2v)\big)(v) \le L_{x_0}(t_2-t_1), ~~ v\in \mathbb S_X, t_2> t_1 \ge 0,
\end{align}
whenever $x+t_1v, x+t_2v \in B_{x_0}$. Next, the convex function $F(x) = \frac {L_{x_0}} 2 \|x\|^2$ satisfies
\begin{align}\label{anti lip hilbert}
\big(F'(x_1)-F'(x_2)\big)(x_1-x_2) \ge L_{x_0}\|x_1-x_2\|^2, ~~ \forall x_1,x_2\in X.
\end{align}
In particular
\begin{align}\label{anti lip}
\big(F'(x+t_2v)-F'(x+t_1v)\big)(v) \ge L_{x_0}(t_2-t_1),  ~~ v\in \mathbb S_X, ~t_2>t_1\ge 0.
\end{align}
Altogether, if $g_y(x) = F(x)-f_y(x)$, then
\begin{align*}
\big(g_y'(x+t_2v)-g_y'(x+t_1v)\big)(v) \stackrel{\eqref{lip prop}\wedge \eqref{anti lip}}{\ge} 0, ~~ v\in \mathbb S_X, ~t_2>t_1\ge 0,
\end{align*}
whenever $x+t_1v, x+t_2v \in B_{x_0}$. This implies that $g_y$ is convex on $B_{x_0}$. It then follows that
\begin{align*}
d_C(x) & = \frac {L_{x_0}} 2 \|x\|^2 -\sup_{y\in C}\Bigg[~\frac {L_{x_0}} 2 \|x\|^2-\|x-y\|\Bigg] = h(x) - \sup_{y\in C}g_y(x)
\end{align*}
is DC on $B_{x_0}$. 
\end{proof}

\begin{remark}
Even in $\R^2$ there are sets for which $d_C$ is not DC everywhere (not even locally DC), as was shown in \cite{BB11}. Thus, the most one could hope for is a locally DC function on $X\setminus C$.
\end{remark}

\medskip

Given $q\in (0,1]$, a norm $\|\cdot\|$ is said to be $q$\emph{-H\"older smooth} at a point $x\in X$ if there exists a constant $K_x\in (0,\infty)$ such that for every $y\in \mathbb S_X$ and every $\tau>0$,
\begin{align*}
\frac{\|x+\tau y\|}{2} + \frac{\|x-\tau y\|}{2} \le 1+K_x\tau^{1+q}.
\end{align*}
If $q=1$ then $(X,\|\cdot\|)$ is said to be \emph{Lipschitz smooth} at $x$.
The spaces $L_p$, $p \ge 2$ are known to be Lipschitz smooth, and in general $L_p$, $p>1$, is $s$-H\"older smooth with $s = \min\{1,p-1\}$. 

\medskip

A Banach space is said to be $p$\emph{-uniformly convex} if for every $x,y\in \mathbb S_X$,
\begin{align*}
1-\left\|\frac{x+y}{2}\right\| \ge L\|x-y\|^p.
\end{align*}
Note that this is similar to assuming that $\delta(\epsilon) = L\epsilon^p$ in \eqref{def uni conv}. The spaces $L_p$, $p>1$, are $r$-uniformly convex with $r = \max\{2,p\}$. 
 
One could ask whether the scheme of proof of Theorem \ref{thm local DC} can be used in a more general setting.

\begin{prop}\label{prop Hilbert}
Let $(X, \|\cdot\|)$ be a Banach space, $C\subseteq X$ a closed set, and fix $x_0\in X\setminus C$ and $y\in C$. Assume that there exists $r_0$ such that $f_y(x) = \|x-y\|$ has a Lipschitz derivative on $B(x_0,r_0)$:
\begin{align}\label{lip of der}
\big\|f_y'(x_1)-f_y'(x_2)\| \le L_{x_0}\|x_1-x_2\|.
\end{align}
Then the norm is Lipschitz smooth on $-y + B_{x_0} = B(x_0-y, r_0)$. If in addition there exists a function $F: X \to \R$ satisfying 
\begin{align}\label{strong conv}
\big(F'(x_1)-F'(x_2)\big)(x_1-x_2) \ge L_{x_0}\|x_1-x_2\|^2, ~~ \forall x_1,x_2\in B(x_0,r_0),
\end{align}
then $(X,\|\cdot\|)$ admits an equivalent norm which is 2-uniformly convex. In particular, if $X=L_p$ then $p=2$.
\end{prop}

\begin{proof}
To prove the first assertion note that \eqref{lip of der} is equivalent to 
\begin{align*}
\|x-y+h\|+\|x-y-h\|-2\|x-y\| \le L_{x_0}\|h\|^2, ~~ x\in B_{x_0}.
\end{align*}
See for example \cite[Prop. 2.1]{Fab85}.

To prove the second assertion, note that a function that satisfies \eqref{strong conv} is also known as \emph{strongly convex}: one can show that \eqref{strong conv} is in fact equivalent to the condition
\begin{align*}
f\left(\frac{x_1+x_2}{2}\right) \le \frac 1 2f(x_1)+\frac 1 2 f(x_2) - C\|x_1-x_2\|^2,
\end{align*}
for some constant $C$. See for example \cite[App. A]{SS07}. This implies that there exists an equivalent norm which is 2-uniformly convex (\cite[Thm 5.4.3]{BV10}).

\end{proof}

\begin{remark}
From \cite{Ara88} it is know that if $F:X\to \R$ satisfies 
\begin{align*}
\big(F'(x_1)-F'(x_2)\big)(v) \ge L\|x_1-x_2\|^2,
\end{align*}
for \emph{all} $x_1,x_2\in X$, and also that $F$ is twice (Fr\'echet) differentiable at one point, then $(X,\|\cdot\|)$ is isomorphic to a Hilbert space.
\end{remark}

\begin{remark}
If we replace the Lipschitz condition by a H\"older condition
\begin{align*}
\big\|f_y'(x_1)-f_y'(x_2)\big\| \le \|x_1-x_2\|^{\beta}, ~~ \beta <1, 
\end{align*}
then in order to follow the same scheme of proof of Theorem \ref{thm local DC}, instead of \eqref{anti lip hilbert}, we would need a function $F$ satisfying
\begin{align*}
\big(F'(x_1)-F'(x_2)\big)(x_1-x_2) \ge \|x_1-x_2\|^{1+\beta}, ~~ x_1,x_2\in B_{x_0}.
\end{align*}
which implies
\begin{align}\label{anti holder}
\big\|F'(x_1)-F'(x_2)\big\| \ge \|x_1-x_2\|^{\beta}, ~~ x_1,x_2\in B_{x_0}.
\end{align}
If $G= (F')^{-1}$, then we get
\begin{align*}
\|Gx_1-Gx_2\| \le \|x_1-x_2\|^{1/\beta}, ~~ x_1,x_2 \in F'(B_{x_0}),
\end{align*}
which can occur only if $G$ is a constant. Hence \eqref{anti holder} cannot hold and the scheme of proof cannot be used if we replace the Lipschitz condition by a H\"older condition.
\end{remark}

\medskip
 
\subsection{DC sets, DC representable sets}
 
\begin{defin}
A set $C$ is is said to be a DC set if $C=A\setminus B$ where $A,B$ are convex.
\end{defin} 
We can also consider the following class of sets.
\begin{defin}
A set $C\subseteq X$ is said to be DC representable if there exists a DC function $f:X\to R$ such that $C = \big\{x\in X: f(x) \le 0\big\}$.
\end{defin}
Note that if $C = A\setminus B$ is a DC set, then we can write $C = \Big\{\mathbbm 1_{B}-\mathbbm 1_{A}+1/2 \le 0\Big\}$, where $\mathbbm 1_A$, $\mathbbm 1_B$ are the indicator functions of $A,B,$ respectively.  Therefore, $C$ is DC representable. Moreover, we have the following.

\begin{thm}[Thach \cite{Th93}]
Assume that $X$ and $Y$ are two Banach space, and $T:Y\to X$ is surjective map with $\mathrm{ker}(T) \neq \emptyset$. Then for any set $M\subseteq X$ there exists a DC representable set $D\subseteq Y$, such that $M= T(D)$.
\end{thm}

Also, the following is known. See \cite{HPT00}.
 
\begin{prop}
If $C$ is a DC representable set, then there exist $A,B \subseteq X\oplus \R$ convex, such that $x\in C \iff (x,x') \in A\setminus B$.
\end{prop}
 
\begin{proof}
Define $g_1(x,x') = f_1(x)-x'$, $g_2(x,x') = f_2(x)-x'$. Let $A = \big\{(x,x') : g_1(x,x')\le 0\big\}$, $B = \big\{(x,x'): g_2(x,x') \le 0\big\}$. Then $x\in C \iff (x,x') \in A\setminus B$.
\end{proof}

In particular, every DC representable set in $X$ is a projection of a DC set in $X\oplus \R$. The following theorem was proved in \cite{TK96}
 
\begin{thm}[Thach-Konno \cite{TK96}]
If $X$ is a reflexive Banach space and $C\subseteq X$ is closed, then $C$ is DC representable.
\end{thm}

This raises the following question.
 
\begin{qst}\label{qst DC}
Is it true that for some classes of spaces, e.g. uniformly convex spaces, there exists $\alpha>0$ such that $d_C^{\alpha}$ is locally DC on $X\setminus C$ whenever $C$ is a DC representable set?
\end{qst}
If the answer to Question \ref{qst DC} is positive, then by the discussion in subsection \ref{sec diff} we could conclude that $N(C)$ has a $\sigma$-porous complement, thus giving an alternative proof of Theorem \ref{thm DMP}. One could also ask Question \ref{qst DC} for DC sets instead of DC representable sets.

\medskip

To end this note, we discuss some simple cases where DC and DC representable sets can be used to study the nearest point problem.

\begin{prop}
Assume that $C = X \setminus \bigcup_{a\in \Lambda}U_a$, where each $U_a$ is an open convex set. Then $d_C$ is locally DC (in fact, locally concave) on $X\setminus C$.
\end{prop}

\begin{proof}
First, it is shown in \cite[Sec. 3]{BF89} that if $a\in \Lambda$, then $d_{X\setminus U_a}$ is concave on $U_a$. Next, it also shown in \cite{BF89} that if $x\in U_a$ then $d_{X\setminus U_a}(x) = d_C(x)$. In particular, $d_C$ is concave on $U_a$. 
\end{proof}

\begin{prop}
Assume that  $C = A\setminus B$ is a closed DC set, and assume $A$ is closed and $B$ is open, then $d_C$ is convex whenever $d_{C}(x) \le d_{A\cap B}$.
\end{prop}
 
\begin{proof}
Since $A = \big(A\setminus B\big) \bigcup B$, we have
$$d_A(x) = \min\big\{d_{A\setminus B}(x), d_{A\cap B}(x)\big\} = \min\big\{d_{C}(x), d_{A\cap B}(x)\big\}.$$
Hence, if $d_C(x) \le d_{A\cap B}(x)$ then $d_C(x) = d_A(x)$ is convex.
\end{proof}
 
\begin{prop}
Assume that $C$ is a DC representable set, i.e., $C = \big\{x\in X: f_1(x)-f_2(x)\le 0\big\}$, and that $f_2(x) =\max_{1\le i \le m}\varphi_i(x)$, where $\varphi_i$ is affine. Then $d_C$ is DC on $X$.
\end{prop}
 
\begin{proof}
Write
\begin{align*}
C & = \Big\{x : f_1(x)-f_2(x)\le 0\Big\}
\\ &= \Big\{x  :  f_1(x)-\max_{1\le i \le m}\varphi_i(x) \le 0\Big\}
\\ & = \Big\{x  : \min_{1\le i \le m}\big(f_1(x)-\varphi_i(x)\big) \le 0\Big\}
\\ & = \bigcup_{i=1}^n \Big\{x : f_1(x)-\varphi_i(x) \le 0\Big\}.
\end{align*}
where the sets $\Big\{x~ : ~ f_1(x)-\varphi_i(x) \le 0\Big\}$ are convex sets. Hence, we have that
\[d_C(x) = \min_{1\le i \le m}d_{C_i}(x),\]
is a minimum of convex sets and therefore by Proposition \ref{prop select} is a DC function.
\end{proof}

In \cite{Cep98} it was shown that if $X$ is superreflexive, then any Lipschitz map is a uniform limit of DC functions. See also \cite[Sec. 5.1]{BV10}. We have the following simple result.

\begin{prop}
If $X$ is separable, then $d_C$ is a limit (not necessarily uniform) of DC functions.
\end{prop} 

\begin{proof}
If $X$ is separable, i.e., there exists a countable $Q = \{q_1,q_2,\dots\}\subseteq X$ with $\bar Q = X$. We have
\begin{align*}
d_C(x) = \inf_{z\in C}\|x-z\| = \inf_{z\in C\cap Q}\|x-z\| = \lim_{n\to \infty} \Big[\min_{z\in C\cap Q_n}\|x-z\|\Big],
\end{align*}
where $Q_n = \{q_1,q_2,\dots,q_n\}$. Again by Proposition \ref{prop select} we have that $\min_{z\in C\cap Q_n}\|x-z\|$ is a DC function as a minimum of convex functions. 
\end{proof}

\section{Conclusion}
Despite many decades of study, the core questions addressed in this note are still far from settled. We hope that our analysis will encourage others to take up the quest, and also to reconsider the related \emph{Chebshev problem} \cite{B07,BV10}.

\end{document}